\documentclass[11pt]{amsart}
\usepackage{lmodern}
\usepackage{amsmath, amsthm, amssymb, amsfonts}
\usepackage[normalem]{ulem}
\usepackage{hyperref}

\usepackage{verbatim} 
\usepackage{longtable}

\usepackage{mathtools}

\usepackage{tikz}
\usetikzlibrary{decorations.pathmorphing}
\tikzset{snake it/.style={decorate, decoration=snake}}

\usepackage{caption}
\usepackage{tikz-cd}
\usetikzlibrary{arrows}

\theoremstyle{plain}
\newtheorem{thm}{Theorem}[section]
\newtheorem{cor}[thm]{Corollary}
\newtheorem{lem}[thm]{Lemma}
\newtheorem{prop}[thm]{Proposition}

\newtheorem{question}[thm]{Question}

\theoremstyle{definition}

\theoremstyle{remark}
\newtheorem{rmk}[thm]{Remark}

\newcommand{\BC}{{\mathbb{C}}}
\newcommand{\BD}{{\mathbb{D}}}

\newcommand{\BP}{{\mathbb{P}}}
\newcommand{\BQ}{{\mathbb{Q}}}

\newcommand{\BZ}{{\mathbb{Z}}}

\newcommand{\CA}{{\mathcal A}}

\newcommand{\CC}{{\mathcal C}}

\newcommand{\CE}{{\mathcal E}}
\newcommand{\CF}{{\mathcal F}}
\newcommand{\CG}{{\mathcal G}}

\newcommand{\CI}{{\mathcal I}}

\newcommand{\CL}{{\mathcal L}}
\newcommand{\CM}{{\mathcal M}}

\newcommand{\CO}{{\mathcal O}}

\newcommand{\CT}{{\mathcal T}}

\DeclareFontFamily{OT1}{rsfs}{}
\DeclareFontShape{OT1}{rsfs}{n}{it}{<-> rsfs10}{}
\DeclareMathAlphabet{\curly}{OT1}{rsfs}{n}{it}

\begin{document}
\title[Derived categories, O'Grady's filtration, and zero-cycles]{Derived categories of $K3$ surfaces, O'Grady's filtration, and zero-cycles on holomorphic symplectic varieties}
\date{\today}

\author{Junliang Shen}
\address{ETH Z\"urich, Department of Mathematics}
\email{junliang.shen@math.ethz.ch}

\author{Qizheng Yin}
\address{Peking University, Beijing International Center for Mathematical Research}
\email{yinqizheng@math.pku.edu.cn}

\author{Xiaolei Zhao}
\address{Northeastern University, Department of Mathematics}
\email{x.zhao@neu.edu}

\begin{abstract}
Moduli spaces of stable objects in the derived category of a $K3$ surface provide a large class of holomorphic symplectic varieties. In this paper, we study the interplay between Chern classes of stable objects and  zero-cycles on holomorphic symplectic varieties which arise as moduli spaces.

First, we show that the second Chern class of any object in the derived category lies in a suitable piece of O'Grady's filtration on the $\mathrm{CH}_0$-group of the $K3$ surface. This solves a conjecture of O'Grady and improves on previous results of Huybrechts, O'Grady, and Voisin. Then we propose a candidate of the Beauville--Voisin filtration on the $\mathrm{CH}_0$-group of the moduli space of stable objects. We discuss its connection with Voisin's recent proposal via constant cycle subvarieties. In particular, we deduce the existence of algebraic coisotropic subvarieties in the moduli space.

Further, for a generic cubic fourfold containing a plane, we establish a connection between zero-cycles on the Fano variety of lines and on the associated $K3$ surface.


\end{abstract}

\baselineskip=14.5pt
\maketitle

\setcounter{tocdepth}{1} 

\tableofcontents
\setcounter{section}{-1}

\section{Introduction}

The purpose of this paper is twofold. On one hand, we study objects in the derived category of a $K3$ surface and their Chern classes. We locate the second Chern classes in the $\mathrm{CH}_0$-group of the $K3$ surface with respect to a filtration introduced by O'Grady, proving and generalizing a conjecture of his. On the other hand, we consider $0$-cycles on holomorphic symplectic varieties which arise as moduli spaces in the derived category. We search for a filtration envisioned by Beauville and Voisin on the $\mathrm{CH}_0$-group of the moduli space, suggesting that it should come from the derived category.

Aspects of derived categories, moduli spaces, and algebraic cycles are brought together.

\subsection{Zero-cycles on $K3$ surfaces}

Let $X$ be a nonsingular projective $K3$ surface. In \cite{BV}, Beauville and Voisin proved that $X$ carries a canonical $0$-cycle class of degree $1$,
\[ [o_X] \in \mathrm{CH}_0(X), \]
where $o_X$ can be taken any point lying on a rational curve in $X$. It has the remarkable property that all intersections of divisor classes in $X$, as well as the second Chern class of $X$, lie in $\mathbb{Z} \cdot [o_X]$.

In \cite{OG2}, O'Grady introduced an increasing filtration $S_\bullet(X)$ on $\mathrm{CH}_0(X)$,
\[S_0(X) \subset S_1(X) \subset \cdots \subset S_i(X) \subset \cdots \subset \mathrm{CH}_0(X),\]
where $S_i(X)$ is the union of $[z] + \mathbb{Z} \cdot [o_X]$ for all effective $0$-cycles $z$ of degree~$i$. In particular, we have
\[S_0(X) = \mathbb{Z} \cdot [o_X].\]
An alternative characterization of $S_\bullet(X)$ via effective orbits is given by Voisin in \cite{V1}.

\subsection{Derived categories}
Let $D^b(X)$ denote the bounded derived category of coherent sheaves on $X$. Given an object $\mathcal{E} \in D^b(X)$, we write
\[v(\mathcal{E}) \in H^0(X, \mathbb{Z}) \oplus H^2(X, \mathbb{Z}) \oplus H^4(X, \mathbb{Z})\]
for the Mukai vector of $\mathcal{E}$, and define
\[d(\mathcal{E}) = \frac{1}{2}\dim \mathrm{Ext}^1(\mathcal{E}, \mathcal{E}) \in \mathbb{Z}_{\geq 0}.\]

An interesting link between the second Chern classes of objects in $D^b(X)$ and the filtration $S_\bullet(X)$ was discovered by Huybrechts and O'Grady. In~\cite{Huy}, \mbox{Huybrechts} showed under certain assumptions\footnote{The assumptions (on the Picard rank of $X$ or on the Mukai vector $v(\mathcal{E})$) were subsequently removed following \cite{V1}.} that if $\mathcal{E} \in D^b(X)$ is a spherical object (and hence $d(\mathcal{E}) = 0$), then
\[c_2(\mathcal{E}) \in \mathbb{Z} \cdot [o_X].\]
Later, O'Grady conjectured\footnote{The statement in \cite[Conjecture 0.4]{OG2} takes a slightly stronger form. However, as is shown in \cite[Proposition 1.3]{OG2}, it is equivalent to what is stated here.} in \cite{OG2} that if $\mathcal{E}$ is a Gieseker-stable sheaf with respect to a polarization $H$ on $X$, then
\[c_2(\mathcal{E}) \in S_{d(\mathcal{E})}(X).\]
He verified the conjecture again under certain assumptions on the Picard rank of $X$ and/or on the Mukai vector $v(\mathcal{E})$. Further, in \cite{V1}, Voisin proved (a generalization of) the conjecture for any simple vector bundle $\mathcal{E}$ on $X$.

Our first result completes the proof of O'Grady's conjecture and generalizes it to arbitrary objects in $D^b(X)$.

\begin{thm} \label{mainthm}
For any object $\mathcal{E} \in D^b(X)$, we have
\[c_2(\mathcal{E}) \in S_{d(\mathcal{E})}(X).\]
\end{thm}

Note that Theorem \ref{mainthm} does not involve any stability condition and should be viewed as a statement purely on the derived category $D^b(X)$. However, the proof uses (slope) stability and ultimately relies on Voisin's proof of the vector bundle case.

Theorem \ref{mainthm} has an important consequence. Let $\widetilde{S}_\bullet(X)$ be the extension of O'Grady's filtration to the Chow ring $\mathrm{CH}^*(X)$ by the trivial filtration on $\mathrm{CH}^0(X)$ and $\mathrm{CH}^1(X)$. In particular, we have
\[\widetilde{S}_0(X) = R^*(X)\]
which is the Beauville--Voisin ring of $X$ generated by divisor classes. Let
\[\Phi : D^b(X) \xrightarrow{\sim} D^b(X')\]
be a derived equivalence between two nonsingular projective $K3$ surfaces. It induces an isomorphism of (ungraded) Chow groups
\[\Phi^{\mathrm{CH}} : \mathrm{CH}^*(X) \xrightarrow{\sim} \mathrm{CH}^*(X').\]
We have the following generalization of Huybrechts' result in \cite{Huy} that $\Phi^{\mathrm{CH}}$ preserves the Beauville--Voisin ring.\footnote{As before, the assumption in \cite{Huy} on the Picard rank can be removed following \cite{V1}.}

\begin{cor} \label{cortomainthm}
The isomorphism $\Phi^{\mathrm{CH}}$ preserves O'Grady's filtration $\widetilde{S}_\bullet$.
\end{cor}

The generality of Theorem \ref{mainthm} also suggests a natural increasing filtration on $D^b(X)$,
\[S_0(D^b(X)) \subset S_1(D^b(X)) \subset \cdots \subset S_i(D^b(X)) \subset \cdots \subset D^b(X),\]
where $S_i(D^b(X))$ consists of objects $\mathcal{E} \in D^b(X)$ with $c_2(\mathcal{E}) \in S_i(X)$. By Corollary~\ref{cortomainthm}, this filtration does not depend on the $K3$ surface $X$ and is ``intrinsic'' to the triangulated category $\mathbf{D} = D^b(X)$.

\subsection{Moduli spaces of stable objects}
The work of Bridgeland \cite{Br, Br1} and Bayer--Macr\`i \cite{BM1, BM} provides a large class of holomorphic symplectic varieties\footnote{They are also referred to as irreducible holomorphic symplectic varieties, or equivalently, hyper-K\"ahler varieties following \cite{B}. In this paper, we emphasize the holomorphic symplectic point of view.} of $K3^{[d]}$-type. They are given as moduli spaces of objects in $D^b(X)$ satisfying certain stability conditions.

Let
\[\mathbf{v} \in H^0(X, \mathbb{Z}) \oplus H^2(X, \mathbb{Z}) \oplus H^4(X, \mathbb{Z})\]
be a primitive algebraic class with Mukai self-intersection $\mathbf{v}^2 >0$. In \cite{Br1}, Bridgeland described a connected component $\mathrm{Stab}^{\dagger}(X)$ of the space of stability conditions on $D^b(X)$, which admits a chamber decomposition depending on~${\bf v}$. When $\sigma \in \mathrm{Stab}^{\dagger}(X)$ is a generic\footnote{The word ``generic'' means away from the walls.} stability condition with respect to~$\mathbf{v}$, there is a nonsingular projective moduli space $M_{\sigma}(\mathbf{v})$ of $\sigma$-stable \mbox{objects} $\mathcal{E} \in D^b(X)$ with Mukai vector $v (\mathcal{E}) = \mathbf{v}$. The moduli space $M_{\sigma}(\mathbf{v})$ only depends on the chamber containing $\sigma$. It is of dimension\footnote{By Riemann--Roch and Serre duality, we have $d(\mathbf{v}) = d(\mathcal{E})$ for any $\mathcal{E} \in M_\sigma(\mathbf{v})$.}
\[2d(\mathbf{v}) = \mathbf{v}^2 + 2 > 2\]
and is holomorphic symplectic by the pairing 
\[\mathrm{Ext}^1(\mathcal{E}, \mathcal{E}) \times \mathrm{Ext}^1(\mathcal{E}, \mathcal{E}) 
\rightarrow \mathrm{Ext}^2(\mathcal{E}, \mathcal{E})\xrightarrow{\mathrm{tr}} \mathbb{C}.\]
When $\sigma$ is in the chamber corresponding to the large volume limit, the moduli space $M_{\sigma}(\mathbf{v})$ recovers the moduli space of Gieseker-stable sheaves with respect to a generic polarization $H$ on $X$.

The following result relates the second Chern classes of objects in $M_{\sigma}(\mathbf{v})$ to the corresponding point classes on $M_{\sigma}(\mathbf{v})$. It was proposed as a conjecture in the first version of this paper, and later shown to hold by Marian and the third author.

\begin{thm}[\cite{MZ}] \label{mainconj} \label{Conj}
Two objects $\mathcal{E}, \mathcal{E}' \in M_{\sigma}(\mathbf{v})$ satisfy
\[[\mathcal{E}] = [\mathcal{E}'] \in \mathrm{CH}_0(M_{\sigma}(\mathbf{v}))\]
if and only if
\[c_2(\mathcal{E}) = c_2(\mathcal{E}') \in \mathrm{CH}_0(X).\]
\end{thm}

\subsection{Beauville--Voisin filtration for zero-cycles}

Our study of $0$-cycles on the moduli spaces $M_{\sigma}(\mathbf{v})$ is motivated by the Beauville--Voisin conjecture for holomorphic symplectic varieties. The conjecture predicts that the Chow ring (with rational coefficients) of a holomorphic symplectic variety admits a multiplicative decomposition; see \cite{B2, V, V2}. Another way to phrase it is the existence of a new filtration on the Chow ring which is opposite to the conjectural Bloch--Beilinson filtration. Recently, rather than proving consequences of the Beauville--Voisin conjecture, much effort has been put to construct this new filtration, which we shall call the Beauville--Voisin filtration.

In the case of a moduli space $M_\sigma(\mathbf{v})$, our previous discussion suggests a natural candidate of the Beauville--Voisin filtration on the Chow group $\mathrm{CH}_0(M_\sigma(\mathbf{v}))$ of $0$-cycles. It is simply given by the restriction of the filtration $S_\bullet(\mathbf{D})$ to $\mathrm{CH}_0(M_\sigma(\mathbf{v}))$. More concretely, we have an increasing filtration
\begin{multline*}
S_0\mathrm{CH}_0(M_\sigma(\mathbf{v})) \subset S_1\mathrm{CH}_0(M_\sigma(\mathbf{v})) \subset \cdots \\
\subset S_i\mathrm{CH}_0(M_\sigma(\mathbf{v})) \subset \cdots \subset \mathrm{CH}_0(M_\sigma(\mathbf{v})),
\end{multline*}
where $S_i\mathrm{CH}_0(M_\sigma(\mathbf{v}))$ is the subgroup spanned by $[\mathcal{E}] \in \mathrm{CH}_0(M_\sigma(\mathbf{v}))$ for all $\mathcal{E} \in M_\sigma(\mathbf{v})$ with $c_2(\mathcal{E}) \in S_i(X)$.

An immediate consequence of Theorem \ref{mainthm} is
\[S_{d(\mathbf{v})}\mathrm{CH}_0(M_\sigma(\mathbf{v})) = \mathrm{CH}_0(M_\sigma(\mathbf{v})) \]
where $2d(\mathbf{v}) = \mathbf{v}^2 + 2$ is the dimension of $M_\sigma(\mathbf{v})$. Moreover, by an argument in \cite{OG2}, the subset
\[\{c_2(\mathcal{E}) : \mathcal{E} \in M_\sigma(\mathbf{v})\} \subset S_{d(\mathbf{v})}(X)\]
equals the full subset of $S_{d(\mathbf{v})}(X)$ of the given degree. In particular, we have
\[S_0\mathrm{CH}_0(M_\sigma(\mathbf{v})) \neq 0.\]
Further, since $S_0(X) = \mathbb{Z} \cdot [o_X]$, Theorem \ref{mainconj} implies that
\[S_0\mathrm{CH}_0(M_\sigma(\mathbf{v})) \simeq \mathbb{Z}.\]
In other words, the moduli space $M_\sigma(\mathbf{v})$ carries a canonical $0$-cycle class of degree $1$, which matches the predictions of the Beauville--Voisin conjecture.


We also show that the filtration $S_\bullet\mathrm{CH}_0$ is independent of birational models or modular interpretations. Hence $S_\bullet\mathrm{CH}_0$ is ``intrinsic'' to $M = M_{\sigma}(\mathbf{v})$ as a moduli space of stable objects in the triangulated category $\mathbf{D} = D^b(X)$.

\begin{prop} \label{presfil}
For any $(X', \sigma', \mathbf{v}')$ such that $M_{\sigma'}(\mathbf{v}')$ is birational to $M_{\sigma}(\mathbf{v})$, the canonical isomorphism of Chow groups
\[\mathrm{CH}_0(M_{\sigma}(\mathbf{v})) \xrightarrow{\sim} \mathrm{CH}_0(M_{\sigma'}(\mathbf{v}'))\]
preserves the filtration $S_\bullet\mathrm{CH}_0$.
\end{prop}

In \cite{V2}, Voisin proposed a filtration on $\mathrm{CH}_0(M)$ for any holomorphic symplectic variety $M$ of dimension $2d$. Given a (closed) point $x \in M$, consider the orbit of $x$ under rational equivalence
\[O_x = \{x' \in M : [x] = [x'] \in \mathrm{CH}_0(M)\}.\]
It is a countable union of constant cycle subvarieties.\footnote{A constant cycle subvariety is a subvariety whose points all share the same class in the $\mathrm{CH}_0$-group of the ambient variety; see \cite{H2, V2}.} We write $\dim O_x$ for the maximal dimension of these subvarieties. There is an increasing~filtration
\begin{multline*}
S_0^V\mathrm{CH}_0(M) \subset S_1^V\mathrm{CH}_0(M) \subset \cdots \subset S_i^V\mathrm{CH}_0(M) \\
\subset \cdots \subset S_d^V\mathrm{CH}_0(M) = \mathrm{CH}_0(M),
\end{multline*}
where $S_i^V\mathrm{CH}_0(M)$ is the subgroup spanned by $[x] \in \mathrm{CH}_0(M)$ for all $x \in M$ with $\dim O_x \geq d - i$. Many questions around the filtration $S_\bullet^V\mathrm{CH}_0(M)$ remain open, among which the existence of algebraic coisotropic subvarieties
\[\begin{tikzcd}
Z_i \arrow[hook]{r}{} \arrow[dashed]{d}{q} & M, \\
B_i
\end{tikzcd}\]
where $Z_i$ is a subvariety of codimension $i$ and the general fibers of $q$ are constant cycle subvarieties (of $M$) of dimension $i$.

The following result constructs algebraic coisotropic varieties and connects the filtrations $S_\bullet\mathrm{CH}_0(M)$ and $S_\bullet^V\mathrm{CH}_0(M)$ in case $M = M_\sigma(\mathbf{v})$.

\begin{thm} \label{coiso}
For $0 \leq i \leq d = d(\mathbf{v})$, 
\begin{enumerate}
\item[(i)] there exists an algebraic coisotropic subvariety $Z_i \dashrightarrow B_i$ of codimension $i$ with constant cycle general fibers;
\item[(ii)] we have $S_i\mathrm{CH}_0(M) \subset S_i^V\mathrm{CH}_0(M)$.
\end{enumerate}
\end{thm}

\subsection{Special cubic fourfolds}

Special cubic $4$-folds are first introduced by Hassett in \cite{Ha} as elements lying in certain Noether--Lefschetz divisors of the moduli space. They have $K3$ surfaces associated to them via either Hodge theory or derived categories; see \cite{AT, Ha, Kuz}.

We focus on the case of a generic cubic $4$-fold $Y$ containing a plane, which lies in the divisor $\CC_8$ of the moduli space in the sense of Hassett~\cite{Ha}. A canonical $K3$ surface $X$ of degree $2$ is constructed and associated to $Y$ by Kuznetsov in \cite{Kuz}. On the other hand, let $F$ denote the Fano variety of lines in $Y$. It is well-known since \cite{BD} that $F$ is holomorphic symplectic of $K3^{[2]}$-type. In \cite{MS}, Macr\`i and Stellari realized $F$ as a moduli space of stable objects in the bounded derived category of twisted coherent sheaves~$D^b(X ,\alpha)$, where~$\alpha$ is an element in the Brauer group of $X$ of order $2$. In particular, the Fano variety $F$ is not birational to the Hilbert scheme of $2$ points on a $K3$ surface.

The following is an analog of Theorem \ref{mainconj} for the Fano varieties of such special cubic $4$-folds.
\begin{thm} \label{cubic4}
Let $Y$ be a generic cubic $4$-fold containing a plane, and let~$X$ be the associated $K3$ surface.
\begin{enumerate}
\item[(i)] There exists a canonical morphism of Chow groups
\[C_Y : \mathrm{CH}_0(F)_\BQ \to \mathrm{CH}_0(X)_\BQ.\]
\item[(ii)] Two lines $l, l'$ in $Y$ satisfy
\[[l] = [l'] \in \mathrm{CH}_0(F)\]
if and only if
\[C_Y([l]) = C_Y([l']) \in \mathrm{CH}_0(X)_\BQ.\]
\end{enumerate}
\end{thm}

Our construction of $C_Y$ uses the motivic decomposition of $F$ as obtained by Shen and Vial in \cite{SV}, plus the geometry of a \mbox{uniruled} divisor in $F$. Since $F$ is isomorphic to a moduli space of stable objects in $D^b(X, \alpha)$, we expect $C_Y$ to be realized geometrically as (the dimension $0$ part of) a natural characteristic class of twisted sheaves.

\subsection{Conventions}

Throughout, we work over the complex numbers $\mathbb{C}$. All varieties are assumed to be (quasi-)projective, and $K3$ surfaces are nonsingular and projective. Equivalences of triangulated categories are $\BC$-linear.

\subsection{Acknowledgements}

We are grateful to Daniel Huybrechts for inspiring the present form of this paper, and to Claire Voisin for a discussion related to Section \ref{VQ}. We also thank Arend Bayer, Zhiyuan Li, Hsueh-Yung Lin, Emmanuele Macr\`i, Alina Marian, Rahul Pandharipande, and Ulrike Rie{\ss} for their interest and for useful discussions.

J.~ S. was supported by grant ERC-2012-AdG-320368-MCSK in the group of Rahul Pandharipande at ETH Z\"urich.

\section{Chern classes and O'Grady's filtration}

In this section, we prove Theorem \ref{mainthm} and Corollary \ref{cortomainthm}.

\subsection{Preliminaries}
We first list a few useful facts. Let $X$ be a $K3$ surface.

\begin{lem}[{\cite[Corollary 1.7]{OG2}}] \label{Lemma1}
Let $\alpha, \alpha' \in \mathrm{CH}_0(X)$.
\begin{enumerate}
\item[(i)] If $\alpha \in S_i(X)$ and $\alpha' \in S_{i'}(X)$, then $\alpha + \alpha' \in S_{i + i'}(X)$. 
\item[(ii)] If $\alpha \in S_i(X)$, then $m\alpha \in S_i(X)$ for any $m \in \mathbb{Z}$. 
\end{enumerate}
\end{lem}

\begin{cor} \label{Cor1}
Let 
\[
  \CF \to \CE \to \CG \to \mathcal{F}[1]
\]
be a distinguished triangle in $D^b(X)$. If two of $c_2(\mathcal{E}), c_2(\mathcal{F}), c_2(\mathcal{G})$ lie in $S_i(X)$ and $S_{i'}(X)$ respectively, then the third lies in $S_{i+i'}(X)$.
\end{cor}

\begin{proof}
By the distinguished triangle, we have
\[
c_2(\CE) = c_2(\CF)+c_2(\CG)+D,
\]
where $D$ is spanned by intersections of divisor classes. Hence $D \in S_0(X)$ by~\cite{BV} and the statement follows immediately from Lemma \ref{Lemma1}.
\end{proof}

We will need the following generalization of a lemma of Mukai \cite[Corollary 2.8]{Muk1}.

\begin{lem}[{\cite[Lemma 2.5]{BB}}]\label{MukLem}
Let
\[
\CF \to \CE \to \CG \to \mathcal{F}[1]
\]
be a distinguished triangle in $D^b(X)$. If $\mathrm{Hom}(\CF , \CG)=0$, then there is an inequality
\[
d(\CF) + d(\CG) \leq d(\CE).
\]
\end{lem}

The following is a direct consequence of Corollary \ref{Cor1} and Lemma \ref{MukLem}.

\begin{prop}\label{proptri}
Let \[
\CF \to \CE \to \CG \to \mathcal{F}[1]
\]
be a distinguished triangle in $D^b(X)$ satisfying $\mathrm{Hom}(\CF , \CG)=0$. If
\[c_2(\CF) \in S_{d(\CF)}(X) \ \ \textrm{and} \ \ c_2(\CG) \in S_{d(\CG)}(X),\]
then
\[c_2(\CE)\in S_{d(\CE)}(X).\]
\end{prop}

We also recall the theorems of Huybrechts and Voisin which serve as the starting point of our proof.

\begin{thm}[{\cite[Theorem 1]{Huy}} and {\cite[Corollary 1.10]{V1}}] \label{Huybrechts}
If $\mathcal{E} \in D^b(X)$ is spherical, \textit{i.e.}, $\mathrm{Ext}^*(\CE, \CE) = H^*(\mathbb{S}^2, \mathbb{C})$, then
\[c_2(\CE) \in S_0(X).\]
\end{thm}

\begin{thm}[{\cite[Theorem 1.9]{V1}}] \label{Voisin}
If $\mathcal{E}$ is a simple vector bundle on $X$,~then
\[c_2(\CE) \in S_{d(\mathcal{E})}(X).\]
\end{thm}

\subsection{Slope-stable sheaves}
From now on, we fix a polarization $H$ on $X$. The slope of a torsion-free sheaf $\CE$ with respect to $H$ is
\[
\mu(\CE) = \frac{c_1(\CE)\cdot H}{\mathrm{rank}(\CE)}.
\]
A torsion-free sheaf $\CE$ is called $\mu$-stable (resp.~$\mu$-semistable) if $\mu(\CF) < \mu(\CE)$ (resp.~$\mu(\CF) \leq \mu(\CE)$) for all subsheaves $\CF \subset \CE$ with $\mathrm{rank}(\CF) < \mathrm{rank}(\CE)$. A $\mu$-stable sheaf is Gieseker-stable (see \cite{HL}), but the converse is not true.

The following proposition proves Theorem \ref{mainthm} for $\mu$-stable sheaves.

\begin{prop}\label{Prop11}
If $\CE$ is torsion-free and $\mu$-stable on $(X, H)$, then
\[
c_2(\CE) \in S_{d(\CE)}(X).
\]
\end{prop}

\begin{proof}
The double dual $\CE^{\vee\vee}$ of $\CE$ is locally free. There is a short exact sequence of sheaves
\begin{equation}\label{equation11}
0 \to \CE \to \CE^{\vee\vee} \to \mathcal{Q} \to 0,
\end{equation}
where $\mathcal{Q}$ is a $0$-dimensional sheaf whose support is of length $l$. A direct calculation yields $d(\mathcal{Q}) \geq l$ and
\[c_2(\mathcal{Q}) \in S_l(X) \subset S_{d(\mathcal{Q})}(X).\] 

Now since $\CE$ is $\mu$-stable, the double dual $\CE^{\vee\vee}$ is also $\mu$-stable and hence simple. Applying Theorem \ref{Voisin}, we find
\[
c_2(\CE^{\vee\vee}) \in S_{d(\CE^{\vee\vee})}(X).
\]
Consider \eqref{equation11} as a distinguished triangle
\[\mathcal{Q}[-1] \to \CE \to \CE^{\vee\vee} \to \mathcal{Q}.\]
Since $\mathcal{Q}$ is $0$-dimensional and $\CE^{\vee\vee}$ is locally free, we have
\[\mathrm{Hom}(\mathcal{Q}[-1], \CE^{\vee\vee})=0.\] Applying Proposition \ref{proptri}, we conclude that $c_2(\CE) \in S_{d(\CE)}(X)$.
\end{proof}

We continue to treat sheaves which can be obtained as iterated extensions of $\mu$-stable sheaves.

Recall that the Mukai vector of a sheaf $\CE$ on $X$ is
\[v(\CE) = \mathrm{ch}(\CE)\sqrt{\mathrm{td}_X}=(\mathrm{rk}(\CE), c_1(\CE), \mathrm{ch}_2(\CE)+\mathrm{rk}(\CE)).\]
By Riemann--Roch and Serre duality, we have
\[
v(\CE)^2 = 2d(\CE) - 2\dim \mathrm{Hom}(\CE, \CE),
\]
where the left-hand side is the Mukai self-intersection.

\begin{prop}\label{Prop2}
Let $\CF$ be torsion-free and $\mu$-stable on $(X, H)$. If $\CE$ is an iterated extension of $\CF$, then
\[
c_2(\CE) \in S_{d(\CE)}(X).
\]
\end{prop}

\begin{proof}
Assume that $\CE$ contains $m$ $\CF$-factors. Then we have
\begin{equation}\label{eqn23}
c_2(\CE) = c_2(\CF^{\oplus m}) = mc_2(\CF) +D,
\end{equation}
where $D$ is spanned by intersections of divisor classes and hence lies in~$S_0(X)$. Combining Lemma \ref{Lemma1} and Proposition \ref{Prop11}, we find 
\begin{equation}\label{eqn32}
c_2(\CE) \in S_{d(\CF)}(X).
\end{equation}

If $\CF$ is spherical, \emph{i.e.}, $v(\CF)^2 = -2$, then $c_2(\CF) \in S_0(X)$ by Theorem \ref{Huybrechts}. By (\ref{eqn23}), we see that $c_2(\CE) \in S_0(X)$, and hence the statement holds.

We may focus on the case $v(\CF)^2 \geq 0$. Then we have
\begin{align*}
2d(\CE) & = v(\CE)^2 +2 \dim \mathrm{Hom}(\CE, \CE)\\
        & = m^2 v(\CF)^2 +2 \dim \mathrm{Hom}(\CE, \CE)\\
        & \geq v(\CF)^2 +2 \\
        & = 2d(\CF),
\end{align*}
where we use that $\CF$ is simple in the last equality. In this case the proposition follows from (\ref{eqn32}).
\end{proof}

\subsection{Torsion-free sheaves}
The next step is to prove Theorem \ref{mainthm} for arbitrary torsion-free sheaves.

The following proposition provides a nice splitting of a $\mu$-semistable vector bundle.

\begin{prop}\label{splitting}
Let $\CE$ be a $\mu$-semistable vector bundle on $(X, H)$. There exists a short exact sequence of sheaves
\begin{equation} \label{ses33}
0 \to \CM \to \CE \to \CG \to 0
\end{equation}
with the following properties:
\begin{enumerate}
\item[(i)] the sheaf $\CM$ is an iterated extension of a $\mu$-stable vector bundle $\CF$;
\item[(ii)] the quotient sheaf $\CG$ is torsion-free;
\item[(iii)] we have $\mathrm{Hom}(\CM, \CG) =0$.
\end{enumerate}
\end{prop}

\begin{proof}
We only need to consider the case when $\CE$ is not $\mu$-stable. First, we can always find a $\mu$-stable sub-vector bundle $\CF \subset \CE$ with $\mu(\CF) = \mu(\CE)$. 

The construction goes as follows. Let $\CF_0$ be any $\mu$-stable subsheaf of $\CE$ with $\mu(\CF_0) = \mu(\CE)$. The double dual $\CF=\CF_0^{\vee\vee}$ is both $\mu$-stable (of the same slope) and locally free, which admits a nontrivial map
\[
i:\CF= \CF_0^{\vee\vee} \rightarrow \CE^{\vee\vee} = \CE.
\]
The map $i$ is injective according to the stability condition. Hence we obtain a short exact sequence of sheaves
\begin{equation}\label{ses1}
0 \to \CF \to \CE \to \CG_0 \to 0.
\end{equation}

\medskip
\noindent{\bf Claim.} The quotient sheaf $\CG_0$ is torsion-free and $\mu$-semistable.

\begin{proof}[Proof of the Claim]
The stability condition ensures that the torsion part of~$\CG_0$ is at most $0$-dimensional. Now assume that there is a short exact sequence of sheaves
\[
0 \rightarrow \mathcal{T} \to \CG_0 \rightarrow \CG^F_0 \rightarrow 0
\]
with $\mathcal{T}$ a nontrivial $0$-dimensional sheaf and $\CG^F_0$ torsion-free. We have a surjective map $\CE \rightarrow \CG^F_0$ given by $\CE \to \CG_0 \to \CG^F_0$ with kernel $\CF'$. It follows that $\CF'$ is a nontrivial extension of the $0$-dimensional sheaf $\mathcal{T}$ by the vector bundle $\CF$, which is a contradiction. This shows that $\CG_0$ is torsion-free.

Since $\mu(\CG_0)= \mu(\CF)$, the $\mu$-semistability follows from a standard argument by considering quotients of $\CG_0$ and comparing slopes.
\end{proof}

If $\mathrm{Hom}(\CF, \CG_0)=0$, then we are done by setting $\CM= \CF$ and $\CG = \CG_0$, and (\ref{ses1}) gives the desired exact sequence. Otherwise, there exists a nontrivial~map
\[
i_1: \CF \rightarrow \CG_0,
\]
which must be injective according to the stability condition. We define $\CG_1$ to be the quotient $\CG_0/\CF$. The same argument as in the Claim implies that  $\CG_1$ is torsion-free and $\mu$-semistable. Hence we obtain a short exact sequence of sheaves
\[
0 \to \CF_1 \to \CE \to \CG_1 \to 0,
\]
where $\CF_1$ is a self-extension of $\CF$. 

If $\mathrm{Hom}(\CF_1, \CG_1) \neq 0$, we can continue this process until we reach the desired exact sequence \eqref{ses33}.
\end{proof}

\begin{rmk}
One may expect a similar splitting for any $\mu$-semistable sheaf via the Jordan--H\"older filtration (for slope stability). However, the difficulty is that there exist nontrivial morphisms between nonisomorphic $\mu$-stable sheaves with the same slope. For example, there is the inclusion
\[
\CI_Z \hookrightarrow \CO_X
\]
with $\CI_Z$ the ideal sheaf of a $0$-dimensional subscheme $Z \subset X$. Here we use a $\mu$-stable locally free factor $\CF$ to avoid this trouble.
\end{rmk}

\begin{prop}\label{torsionfree}
If $\CE$ is a torsion-free sheaf on $X$, then
\[
c_2(\CE)\in S_{d(\CE)}(X).
\]
\end{prop}

\begin{proof}
We proceed by induction on the rank of $\CE$. If $\mathrm{rank}(\CE)=1$, then $\CE$ is $\mu$-stable, and Proposition \ref{Prop11} gives the base case of the induction.

Now assume that $\CE$ is torsion-free of rank $r>0$. If $\CE$ is not $\mu$-semistable, then by the Harder--Narasimhan filtration (for slope stability), we have a short exact sequence of sheaves
\[
0 \to \CF \to \CE \to \CG \to 0.
\]
Here $\CF$ and $\CG$ are nonzero and torsion-free, and the slope of every $\mu$-stable factor of $\CF$ is greater than the slope of any $\mu$-stable factor of~$\CG$. In particular, we have $\mathrm{rank}(\CF)<\mathrm{rank}(\CE)$, $\mathrm{rank}(\CG)<\mathrm{rank}(\CE)$, and $\mathrm{Hom}(\CF, \CG) =0$. The induction hypothesis yields
\[
c_2(\CF) \in S_{d(\CF)}(X) \ \ \textrm{and} \ \ c_2(\CG) \in S_{d(\CG)}(X).
\]
Applying Proposition \ref{proptri}, we find $c_2(\CE) \in S_{d(\CE)}(X)$.

It remains to treat the case when $\CE$ is $\mu$-semistable. By the same argument as in Proposition \ref{Prop11}, it suffices to prove Theorem \ref{mainthm} for $\CE^{\vee\vee}$, which is a $\mu$-semistable locally free sheaf satisfying $\mathrm{rank}(\CE)= \mathrm{rank}(\CE^{\vee\vee})$.

Hence we may assume $\CE$ to be $\mu$-semistable and locally free. We apply Proposition \ref{splitting} to $\CE$. Either $\CE$ is an iterated extension of some $\mu$-stable sheaf $\CF$, or the extension (\ref{ses33}) is nontrivial. In the first case, the statement of the proposition holds by Proposition \ref{Prop2}. In the second case, the induction hypothesis and Proposition \ref{proptri} complete the proof.   
\end{proof}

\subsection{Torsion sheaves}

Theorem \ref{mainthm} for torsion sheaves is essentially proven in \cite{OG2}. We begin by recalling the following criterion of O'Grady.

\begin{lem}[{\cite[Claim 0.2]{OG2}}]\label{OGlem}
Let $C$ be an irreducible nonsingular curve of genus $g$, and let $f: C \rightarrow X$ be a nonconstant map. Then
\[
f_\ast \mathrm{CH}_0(C) \subset S_g(X).
\]
\end{lem}

Let $\CE$ be a pure $1$-dimensional torsion sheaf on $X$ with Mukai vector
\[
v(\CE)= (0, l, s) \in H^0(X, \mathbb{Z}) \oplus H^2(X, \mathbb{Z}) \oplus H^4(X, \mathbb{Z}).
\]
By Lemma~\ref{OGlem}, we have at worst 
\[
c_2(\CE) \in S_g(X),
\]
where $g= \frac{1}{2}l^2+1$ is the arithmetic genus of the support curve of $\CE$.

On the other hand, we find
\[
d(\CE) = \frac{1}{2}v(\CE)^2+ \dim \mathrm{Hom}(\CE, \CE) \geq \frac{1}{2}l^2+1 = g.
\]
Hence for any pure $1$-dimensional sheaf $\CE$, we have $c_2(\CE) \in S_{d{(\CE)}}(X)$.

Now we can prove Theorem \ref{mainthm} for arbitrary sheaves.

\begin{prop}\label{base}
If $\CE$ is a coherent sheaf on $X$, then 
\[
c_2(\CE) \in S_{d({\CE})}(X).
\]
\end{prop}

\begin{proof}
Given a torsion sheaf $\CT$, there is a short exact sequence of sheaves
\[
0 \to \CT_0 \to \CT \to \CT_1 \to 0,
\]
where $\CT_0$ is $0$-dimensional and $\CT_1$ is pure and $1$-dimensional. Clearly 
\[
\mathrm{Hom}(\CT_0, \CT_1)=0.
\]
By the discussion above, we have
\[
c_2(\CT_0)\in S_{d(\CT_0)}(X) \ \ \textrm{and} \ \ c_2(\CT_1) \in S_{d(\CT_1)}(X),
\]
Applying Proposition \ref{proptri}, we find $c_2(\CT) \in  S_{d{(\CT)}}(X)$ which proves the statement for torsion sheaves.

Let $\CE$ be an arbitrary sheaf. There is a short exact sequence of sheaves
\[
0 \to \CT \to \CE \to \CF \to 0
\]
with $\CT$ torsion and $\CF$ torsion-free. In particular, we have $\mathrm{Hom}(\CT, \CF)=0$. Since the statement of the proposition holds for both $\CT$ and $\CF$, we conclude by Proposition \ref{proptri} that $c_2(\CE) \in S_{d(\CE)}(X)$.
\end{proof}

\subsection{Proof of Theorem \ref{mainthm} and Corollary \ref{cortomainthm}}
Given a bounded complex $\CE \in D^b(X)$, we define its length by
\[
\ell(\CE)= \mathrm{max}\{ |i-j|: h^i(\CE)\neq 0, \, h^j(\CE) \neq 0\}.
\]
Clearly $\ell(\CE)=0$ if and only if $\CE$ is a (shifted) sheaf. 

\begin{proof}[Proof of Theorem \ref{mainthm}]
We proceed by induction on $\ell(\CE)$. Proposition \ref{base} provides the base case of the induction.

Now consider a bounded complex $\CE \in D^b(X)$. Let $m$ be the largest integer such that $h^m(\CE) \neq 0$. There is a standard distinguished triangle 
\[
\CF \to \CE \to \CG \to \CF[1].
\]
Here $\CG$ is the shifted sheaf $h^m(\CE)[-m]$ and $\CF \in D^b(X)$ is the truncated complex $\tau^{\leq{m-1}}\CE$ which satisfies
\[
\ell(\CF) < \ell(\CE).
\]
By the induction hypothesis, we have
\[
c_2(\CF) \in S_{d(\CF)}(X) \ \ \textrm{and} \ \ c_2(\CG) \in S_{d(\CG)}(X).
\]

Since $\CF$ is concentrated in degrees $<m$ and $\CG$ in degree~$m$, we have $\mathrm{Hom}(\CF, \CG)=0$. Applying Proposition \ref{proptri}, we find $c_2(\CE) \in S_{d(\CE)}(X)$. The proof of Theorem \ref{mainthm} is complete.
\end{proof}

Let $X$ and $X'$ be two $K3$ surfaces. Suppose there is a derived equivalence
\[
\Phi: D^b(X) \xrightarrow{\sim} D^b(X')
\]
with Fourier--Mukai kernel $\CF \in D^b(X \times X')$. The induced isomorphism of (ungraded) Chow groups
\[
\Phi^{\mathrm{CH}}: \mathrm{CH}^\ast(X) \xrightarrow{\sim} \mathrm{CH}^\ast(X')
\]
is given by the correspondence\footnote{Here the square root is taken with respect to the canonical classes $[o_X] \in \mathrm{CH}_0(X)$ and $[o_{X'}] \in \mathrm{CH}_0(X')$; see \cite{Huy}.}
\[v^{\mathrm{CH}}(\mathcal{F}) = \mathrm{ch}(\CF)\sqrt{\mathrm{td}_{X\times X'}} \in \mathrm{CH}^\ast(X \times X').\]

Recall the following theorem of Huybrechts and Voisin.

\begin{thm}[{\cite[Theorem 2]{Huy}} and {\cite[Corollary 1.10]{V1}}] \label{Huybrechts2}
The isomorphism $\Phi^{\mathrm{CH}}$ preserves the Beauville--Voisin ring.
\end{thm}

\begin{proof}[Proof of Corollary \ref{cortomainthm}]
Since $\Phi$ is a derived equivalence, we only need to prove that
\[
\Phi^{\mathrm{CH}}(\widetilde{S}_i(X)) \subset \widetilde{S}_i(X').
\]
Since $\Phi^{\mathrm{CH}}$ preserves the Beauville--Voisin ring by Theorem \ref{Huybrechts2}, it suffices to show that for any effective $0$-cycle $z = x_1 + \cdots + x_i$, we have
\begin{equation}\label{lasteqn}
\Phi^{\mathrm{CH}}([z]) \in \widetilde{S}_i(X').
\end{equation}
Further, we may assume $x_1, \dots, x_i$ distinct, since multiplicities result in $[z] \in S_{i'}(X)$ for some $i' < i$.

Let $\CE$ be the direct sum of skyscraper sheaves \[\BC_{x_1} \oplus \cdots \oplus \BC_{x_i}.\]
Then $c_2(\CE)=[z]$ and $d(\Phi(\CE))= d(\CE) = i$. Applying Theorem \ref{mainthm}, we find
\[
c_2(\Phi(\CE))\in S_i(X'),
\]
which implies (\ref{lasteqn}).
\end{proof}

\section{Moduli spaces of stable objects} \label{msso}

In this section, we discuss the Beauville--Voisin conjecture in the case of moduli spaces of stable objects. We prove Proposition \ref{presfil} and Theorem \ref{coiso}.

\subsection{Independence of modular interpretations}

The proof of Proposition \ref{presfil} uses Bayer and Macr\`i's work \cite{BM1, BM} on the birational transforms of moduli spaces of stable objects.

Let $X$ be a $K3$ surface. Recall that given a primitive Mukai vector\footnote{In this section we always assume $\mathbf{v}^2 > 0$, so that the moduli space $M_\sigma(\mathbf{v})$ is of dimension $> 2$. This assumption is crucial in Theorem \ref{BayerMacri}. See Section \ref{dimtwo} for a discussion of the dimension $2$ case.} $\mathbf{v}$ with $\mathbf{v}^2 > 0$, and a generic stability condition $\sigma \in \mathrm{Stab}^{\dag}(X)$ with respect to~$\mathbf{v}$, there is a moduli space $M_\sigma(\mathbf{v})$ of $\sigma$-stable objects in $D^b(X)$.

Bayer and Macr\`i realized all holomorphic symplectic birational models of~$M_\sigma(\mathbf{v})$ as other moduli spaces of stable objects. Their following theorem describes the procedure concretely.

\begin{thm}[{\cite[Corollary 1.3]{BM1}}] \label{BayerMacri}
With the notation above, let $(X', \sigma', \mathbf{v}')$ be another triple. The moduli spaces $M_\sigma(\mathbf{v})$ and $M_{\sigma'}(\mathbf{v}')$ are birational if and only if there exists a derived (anti-)equivalence 
\[
\Phi: D^b(X) \xrightarrow{\sim} D^b(X')
\]
which sends $\mathbf{v}$ to $\mathbf{v}'$ and induces an isomorphism
\[\Sigma: U \xrightarrow{\sim} U'\]
between two nonempty open subsets $U \subset M_\sigma(\mathbf{v})$ and $U' \subset M_{\sigma'}(\mathbf{v}')$.\footnote{Here $\Sigma$ sends an object $\mathcal{E} \in U$ to $\Phi(\mathcal{E}) \in U'$.}
\end{thm}

It is well-known that the $\mathrm{CH}_0$-group is invariant under birational transforms of nonsingular projective varieties.\footnote{In \cite{R}, Rie{\ss} proved that birational holomorphic symplectic varieties have isomorphic Chow rings.} The statement can be made slightly more precise.

\begin{lem} \label{point}
Let $f: V \dashrightarrow V'$ be a birational map between nonsingular projective varieties, and let
\[f_*: \mathrm{CH}_0(V) \xrightarrow{\sim} \mathrm{CH}_0(V')\]
be the induced isomorphism of Chow groups. Then for any point $x \in V$, there exists a point $x' \in V'$ such that
\[[x'] = f_*([x]) \in \mathrm{CH}_0(V').\]
\end{lem}

\begin{proof}
Consider a resolution
\[\begin{tikzcd}
& \widetilde{V} \arrow{dl}[swap]{p} \arrow{dr}{q} \\
V \arrow[dashed]{rr}{f} & & V'
\end{tikzcd}\]
with $\widetilde{V}$ nonsingular and projective. Then $f_*$ is realized as $q_*p^*$. By weak factorization, both $p$ and $q$ can be taken a sequence of blow-ups and blow-downs with nonsingular centers. We are reduced to the case of a blow-up, for which the statement is obvious.
\end{proof}

Let $(X, \sigma, \mathbf{v})$ and $(X', \sigma', \mathbf{v}')$ be such that $M_\sigma(\mathbf{v})$ and $M_{\sigma'}(\mathbf{v}')$ are birational. By Theorem \ref{BayerMacri}, a derived (anti-)equivalence
\[\Phi: D^b(X) \xrightarrow{\sim} D^b(X')\]
induces a birational map
\[\Sigma: M_\sigma(\mathbf{v}) \dashrightarrow M_{\sigma'}(\mathbf{v}'),\]
which identifies two nonempty open subsets $U \subset M_\sigma(\mathbf{v})$ and $U' \subset M_{\sigma'}(\mathbf{v}')$. By further composing with $R\mathcal{H}om(-, \mathcal{O}_X)$, we may assume that $\Phi$ is a derived equivalence.

Let $\mathcal{E}$ be an object in $M_\sigma(\mathbf{v})$. By Lemma \ref{point}, there exists an object $\mathcal{F}$ in $M_{\sigma'}(\mathbf{v}')$ such that
\begin{equation} \label{AAAA}
[\mathcal{F}] = \Sigma_*([\mathcal{E}]) \in \mathrm{CH}_0(M_{\sigma'}(\mathbf{v}')).
\end{equation}

\begin{lem} \label{funct}
With the notation above, for any pair of objects $\mathcal{E} \in M_\sigma(\mathbf{v})$ and $\mathcal{F} \in M_{\sigma'}(\mathbf{v}')$ satisfying \eqref{AAAA}, we have\footnote{Recall that $v^{\mathrm{CH}}(\mathcal{E}) = \mathrm{ch}(\mathcal{E})\sqrt{\mathrm{td}_X} \in \mathrm{CH}^*(X)$ for $\mathcal{E} \in D^b(X)$.}
\[v^{\mathrm{CH}}(\mathcal{F}) = \Phi^{\mathrm{CH}}(v^{\mathrm{CH}}(\mathcal{E})) \in \mathrm{CH}^*(X').\]
\end{lem}

\begin{proof}
Since any class in $\mathrm{CH}_0(M_\sigma(\mathbf{v}))$ is supported on $U$, we may write
\[[\mathcal{E}] = \sum_ja_j[\mathcal{E}_j] \in \mathrm{CH}_0(M_\sigma(\mathbf{v}))\]
for some $\mathcal{E}_j \in U$. By the (quasi-)universal family on $M_\sigma(\mathbf{v}) \times X$, we have
\begin{equation} \label{GGGG}
v^{\mathrm{CH}}(\mathcal{E}) = \sum_ja_jv^{\mathrm{CH}}(\mathcal{E}_j) \in \mathrm{CH}^*(X).
\end{equation}
On the other hand, it is clear from the definition that
\[
[\mathcal{F}] = \sum_ja_j[\Sigma(\mathcal{E}_j)] = \sum_ja_j[\Phi(\mathcal{E}_j)] \in \mathrm{CH}_0(M_{\sigma'}(\mathbf{v}')).
\]
By the (quasi-)universal family on $M_{\sigma'}(\mathbf{v}') \times X'$, we have
\begin{equation} \label{HHHH}
v^{\mathrm{CH}}(\mathcal{F}) = \sum_ja_jv^{\mathrm{CH}}(\Phi(\mathcal{E}_j)) \in \mathrm{CH}^*(X').
\end{equation}
Combining \eqref{GGGG} and \eqref{HHHH} and using the commutative diagram
\[\begin{tikzcd}
D^b(X) \arrow{r}{\Phi} \arrow{d}[swap]{v^{\mathrm{CH}}} & D^b(X') \arrow{d}{v^{\mathrm{CH}}} \\
\mathrm{CH}^*(X) \arrow{r}{\Phi^{\mathrm{CH}}} & \mathrm{CH}^*(X'),
\end{tikzcd}\]
we find
\[
v^{\mathrm{CH}}(\mathcal{F}) = \Phi^{\mathrm{CH}}(\sum_ja_jv^{\mathrm{CH}}(\mathcal{E}_j)) \\
= \Phi^{\mathrm{CH}}(v^{\mathrm{CH}}(\mathcal{E})) \in \mathrm{CH}^*(X'). \qedhere
\]
\end{proof}

\begin{proof}[Proof of Proposition \ref{presfil}]
Let $\mathcal{E}$ be an object in $M_\sigma(\mathbf{v})$ such that
\[
c_2(\mathcal{E}) \in S_i(X).
\]
Equivalently, we have
\[v^{\mathrm{CH}}(\mathcal{E}) \in \widetilde{S}_i(X).\]
By Lemma \ref{point}, there exist an object $\mathcal{F}$ in $M_{\sigma'}(\mathbf{v}')$ satisfying
\[[\mathcal{F}] = \Sigma_*([\mathcal{E}]) \in \mathrm{CH}_0(M_{\sigma'}(\mathbf{v}')).\]
Applying Lemma \ref{funct}, we find
\[
v^{\mathrm{CH}}(\mathcal{F}) = \Phi^{\mathrm{CH}}(v^{\mathrm{CH}}(\mathcal{E})) \in \mathrm{CH}^*(X').
\]
Since $\Phi^{\mathrm{CH}}$ preserves the filtration $\widetilde{S}_\bullet$ by Corollary \ref{cortomainthm}, we conclude that
\[v^{\mathrm{CH}}(\mathcal{F}) \in \widetilde{S}_i(X'),\]
or equivalently,
\[c_2(\mathcal{F}) \in S_i(X').\]
The proposition then follows from the definition of $S_\bullet\mathrm{CH}_0$.
\end{proof}

\subsection{The Beauville--Voisin filtration}

As stated in Theorem \ref{coiso}, we compare two proposed filtrations on the $\mathrm{CH}_0$-group of a moduli space of stable objects.

Let $X$ be a $K3$ surface, let $M = M_\sigma(\mathbf{v})$ be a moduli space of stable objects in $D^b(X)$ of dimension $2d = 2d(\mathbf{v})$, and let $X^{[d]}$ be the Hilbert scheme of $d$ points on $X$. Consider the incidence
\[R = \{(\mathcal{E}, \xi) \in M \times X^{[d]} : c_2(\mathcal{E}) = [\mathrm{Supp}(\xi)] + c[o_X] \in \mathrm{CH}_0(X)\},\]
where $\mathrm{Supp}(\xi)$ is the support of $\xi$ and $c \in \BZ$ is a constant determined by the Mukai vector $\mathbf{v}$. This incidence has already appeared in \cite{OG2, V1}.

A standard argument using Hilbert schemes shows that $R$ is a countable union of Zariski-closed subsets of $M \times X^{[d]}$. Let
\[p_M : R \to M \ \ \textrm{and} \ \ p_{X^{[d]}} : R \to X^{[d]}\]
denote the two projections. By Theorem \ref{mainconj} for $X^{[d]}$ or an explicit calculation, all points on the same fiber of $p_M$ have the same class in $\mathrm{CH}_0(X^{[d]})$. Similarly, by Theorem \ref{mainconj} for $M$, all points on the same fiber of $p_{X^{[d]}}$ have the same class in $\mathrm{CH}_0(M)$.

An important consequence of Theorem \ref{mainthm} is that $p_M$ is dominant. Then, by the argument in \cite[Proposition 1.3]{OG2} (see also \cite[Corollary 3.4]{V1}), we also know that $p_{X^{[d]}}$ is dominant. More precisely, there exists a component $R_0 \subset R$ which dominates both $M$ and $X^{[d]}$. Note that $M$ and $X^{[d]}$ have the same dimension.

Further, up to taking hyperplane sections\footnote{This step is unnecessary: by \cite[Theorem 1.3]{V2}, the orbit of a very general point on~$M$ under rational equivalence is discrete.}, we may assume that $R_0$ is generically finite over both $M$ and $X^{[d]}$. To summarize, we have a diagram
\begin{equation} \label{diag}
\begin{tikzcd}
& R_0 \arrow{dl}[swap]{p_M} \arrow{dr}{p_{X^{[d]}}} \\
U \subset M \phantom{\supset U} & & \phantom{V \subset} X^{[d]} \supset V,
\end{tikzcd}
\end{equation}
where $U \subset M$ and $V \subset X^{[d]}$ are nonempty open subsets over which $p_M$ and $p_{X^{[d]}}$ are finite.

We recall two density results on $X$ and $X^{[d]}$.

\begin{lem}[{\cite[Lemma 2.3]{V1}}; see also {\cite[Lemma 6.3]{H2}}] \label{dense1}
The union of constant cycle curves in $X$ is Zariski-dense.
\end{lem}

\begin{lem}[{\cite[Theorem 1.2]{Mac}}; see also {\cite[Lemma 3.5]{V1}}] \label{dense2}
For any point $\xi \in X^{[d]}$, its orbit under rational equivalence $O_\xi \subset X^{[d]}$ is Zariski-dense.
\end{lem}

\begin{proof}[Proof of Theorem \ref{coiso}]
Given $d - i$ constant cycle curves in $X$ labelled as $C_{i + 1}, C_{i + 2}, \ldots, C_d$, we consider the rational map
\[
\phi : X^{[i]} \times C_{i + 1} \times C_{i + 2} \times \cdots \times C_d \dashrightarrow X^{[d]}
\]
which (generically) sums up the points on the factors. By Lemma \ref{dense1}, the union of $\mathrm{Im}(\phi)$ for all choices of constant cycles curves is Zariski-dense in~$X^{[d]}$. In particular, there exists such $\phi$ whose image meets $V \subset X^{[d]}$.

Let $\phi': Z \dashrightarrow R_0$ denote the pull-back of $\phi$ via $p_{X^{[d]}}$.\footnote{If $Z$ is not irreducible, we may take an irreducible component of $Z$.} We have a diagram
\[\begin{tikzcd}
Z \arrow[dashed]{r}{\phi'} \arrow[dashed]{d}[swap]{p'} & R_0 \arrow{d}{p_{X^{[d]}}} \\
X^{[i]} \times C_{i + 1} \times C_{i + 2} \times \cdots \times C_d \arrow[dashed]{r}{\phi} & X^{[d]}.
\end{tikzcd}\]
Again by Lemma \ref{dense1}, we may assume that $\phi'(Z)$ meets $p_M^{-1}(U) \subset R_0$.

Let $q : Z \dashrightarrow X^{[i]}$ denote the composition of $p'$ and the projection to~$X^{[i]}$. For a general point $\xi \in X^{[i]}$, consider the fiber $Z_{\xi} \subset Z$. By construction, the image
\[p_M(\phi'(Z_\xi)) \subset M\]
consists of objects in $M$ with constant second Chern class. By Theorem \ref{mainconj}, this gives a constant cycle subvariety in~$M$. The dimension of $p_M(\phi'(Z_\xi))$ is $d - i$ since $p_M$ and $p_{X^{[d]}}$ are finite over $U$ and $V$.

We have shown that the image $p_M(\phi'(Z))$ is generically covered by constant cycle subvarieties of dimension $d - i$. We conclude by \cite[Theorem~0.7]{V2} that $p_M(\phi'(Z))$ is algebraically coisotropic of codimension~$d - i$ with constant cycle general fibers. This proves part (i) of the theorem.

For part (ii), let $\mathcal{E}$ be an object in $M$ such that $c_2(\mathcal{E}) \in S_i(X)$. By definition, there exists a point $\xi_0 \in X^{[i]}$ satisfying
\[
c_2(\mathcal{E}) = \mathrm{Supp}(\xi_0) + (d - i)[o_X] + c[o_X] \in \mathrm{CH}_0(X).
\]
Applying Lemma \ref{dense2} to $\xi_0 \in X^{[i]}$, we may further assume that $p_M(\phi'(Z_{\xi_0}))$ is well-defined and is of dimension $d - i$.

By construction, the subvariety $p_M(\phi'(Z_{\xi_0}))$ consists of objects in $M$ whose second Chern class equals $c_2(\mathcal{E})$. By Theorem~\ref{mainconj}, it is a subvariety of dimension $d - i$ in the orbit $O_{\mathcal{E}} \subset M$. We conclude that $[\mathcal{E}] \in S_i^V\mathrm{CH}_0(M)$, which proves part (i) of the theorem.
\end{proof}

\begin{rmk} \label{densi}
Our proof relies on the Zariski density of subvarieties of maximal dimension in an orbit of $X^{[d]}$. If one could prove such density for $M$, then an argument using \cite[Theorem 2.1]{V1} would yield the other inclusion
\[S_i^V\mathrm{CH}_0(M) \subset S_i\mathrm{CH}_0(M).\]
\end{rmk}

\section{Special cubic fourfolds}

We prove Theorem \ref{cubic4} in this section. The proof uses the work of Shen and Vial in \cite[Part 3]{SV} on the Chow motive of the Fano variety of lines in a cubic $4$-fold.


\subsection{The Shen--Vial decomposition}
Let $Y$ be a cubic $4$-fold and let $F$ be its Fano variety of lines. For our purpose, we shall restrict to the case of $0$-cycles on $F$.

Given a general line $l$ in $Y \subset \BP^5$, there exists a unique plane $\mathbb{P}_l^2 \subset \mathbb{P}^5$ which is tangent to $Y$ along $l$. We have a proper intersection in $\mathbb{P}^5$, 
\[
\mathbb{P}_l^2 \cdot Y = 2l + l',
\]
where $l'$ is another line in $Y$. In \cite{V0}, Voisin introduced the rational~map
\[\varphi : F \dashrightarrow F\]
of degree $16$ which sends $l$ to $l'$.

Consider the action of $\varphi$ on $\mathrm{CH}_0(F)$ via the push-forward
\[
\varphi_\ast: \mathrm{CH}_0(F) \rightarrow \mathrm{CH}_0(F).
\]

\begin{thm}[{\cite[Theorem 21.9]{SV}}] \label{eigen}
There is a decomposition
\[
\mathrm{CH}_0(F)_\BQ=  \mathrm{CH}_0(F)_{(0)} \oplus \mathrm{CH}_0(F)_{(2)} \oplus \mathrm{CH}_0(F)_{(4)}
\]
into eigenspaces of $\phi_*$ with eigenvalues $1$, $-2$, and $4$ respectively.
\end{thm}

Define the incidence variety $I \subset F \times F$ to be
\[
I=\{ (l_1, l_2) \in F \times F: l_1 \cap l_2 \neq \emptyset \subset Y\}.
\]
The correspondence
\[[I]_\ast : \mathrm{CH}_0(F) \rightarrow \mathrm{CH}^2(F)\] sends the class of a line $[l] \in \mathrm{CH}_0(F)$ to $[S_l] \in \mathrm{CH}_2(F)$, where $S_l$ is the surface formed by lines meeting $l$. 


The following proposition gives another description of $\mathrm{CH}_0(F)_{(4)}$ via $[I]_\ast$.

\begin{prop}[{\cite[Definition 20.1 and Theorem 21.9]{SV}}] \label{CH4}
We have
\[
\mathrm{CH}_0(F)_{(4)} = \mathrm{Ker}( [I]_\ast )_\BQ.
\]
\end{prop}

By \cite{SV, V}, the group $\mathrm{CH}_0(F)_{(0)}$ is spanned by a canonical $0$-cycle class of degree $1$,
\[
\mathrm{CH}_0(F)_{(0)} = \BQ \cdot [o_F].
\]
where $o_F$ can be taken any point lying on a rational surface in $F$. Moreover, all $0$-dimensional intersections of divisor classes and Chern classes of $F$ are multiples of $[o_F]$. 

The next lemma concerns the self-intersection of the surface $S_l$.

\begin{lem}\label{SVrelation}
For a line $l \subset Y$, we have
\[
[S_l]^2 = 6[o_F]+ \varphi_\ast [l] - 2[l] \in \mathrm{CH}_0(F).
\]
\end{lem}

\begin{proof}
If the point $l \in F$ lies in the open subset where the rational map $\varphi$ is well-defined, then $\varphi_\ast [l] = [\varphi(l)]$, and the statement is a special case of \cite[Proposition 20.7 (i)]{SV}. The same result for an arbitrary $l \in F$ follows from a standard ``spreading out'' argument.
\end{proof}



Now we prove the main result of this subsection. For a line $l \subset Y$, the decomposition in Theorem \ref{eigen} yields
\begin{equation}\label{decompeqn}
[l] = [l]_{(0)}+ [l]_{(2)}+ [l]_{(4)} \in \mathrm{CH}_0(F)_\BQ,
\end{equation}
where $[l]_{(0)}=[o_F]$.

\begin{thm}\label{thm7}
Two lines $l$ and $l'$ in $Y$ satisfy
\[
[l] = [l'] \in \mathrm{CH}_0(F)
\]
if and only if
\[[l]_{(2)} = [l']_{(2)} \in \mathrm{CH}_0(F)_{(2)}.\]
\end{thm}

\begin{proof}
We only need to prove the ``if'' part. The decomposition (\ref{decompeqn}) for $[l]$ and $[l']$ gives
\[
[l]- [l'] = ([l]_{(2)}-[l']_{(2)})+ ([l]_{(4)}-[l']_{(4)}) \in \mathrm{CH}_0(F)_\BQ.
\]
If $[l]_{(2)}= [l']_{(2)} \in \mathrm{CH}_0(F)_{(2)}$, then we have 
\begin{equation} \label{cubiceqn1}
[l] - [l'] \in \mathrm{CH}_0(F)_{(4)}.
\end{equation}
Applying Theorem \ref{eigen}, we find
\begin{equation}\label{relation1}
\varphi_\ast\left( [l] - [l']\right) = 4\left([l] - [l']\right) \in \mathrm{CH}_0(F)_\BQ.
\end{equation}

On the other hand, by Proposition \ref{CH4}, the condition \eqref{cubiceqn1} implies that 
\[
[I]_\ast\left([l] - [l'] \right) = [S_l]- [S_{l'}] =0 \in  \mathrm{CH}^2(F)_\BQ.
\] 
Hence $[S_l]^2 = [S_{l'}]^2 \in \mathrm{CH}_0(F)_\BQ$, and by Lemma \ref{SVrelation} we have
\begin{equation}\label{relation1'}
\varphi_\ast \left( [l] - [l'] \right) = 2\left( [l] - [l']\right) \in \mathrm{CH}_0(F)_\BQ.
\end{equation}
Combining (\ref{relation1}) and (\ref{relation1'}), we find $[l]=[l'] \in \mathrm{CH}_0(F)_\BQ$, and we conclude by Ro{\u\i}tman's theorem \cite{Roi} that $[l]=[l'] \in \mathrm{CH}_0(F)$.
\end{proof}

\subsection{Cubic fourfolds containing a plane}
Let $Y$ be a generic cubic $4$-fold which contains a plane $\mathbb{P}^2 \subset \mathbb{P}^5$. We first review the construction of the $K3$ surface associated to $Y$ in \cite[Section 4]{Kuz}.

Let $\widetilde{Y} \subset \mathrm{Bl}_{\mathbb{P}^2}(\BP^5)$ denote the $4$-fold obtained by blowing up the plane~$\mathbb{P}^2$ in $Y$. The rational map $Y \dashrightarrow \mathbb{P}^2$ given by the linear projection from $\mathbb{P}^2 \subset Y$ induces a morphism
\[f: \widetilde{Y} \rightarrow \BP^2,\]
which is a quadric fibration. We call $l \subset \widetilde{Y}$ a line if it is either a line in the exceptional divisor over $Y$ or the strict transform of a line in $Y$. 

We define $\CM$ to be the moduli space of lines in $\widetilde{Y}$ contained in the fibers of $f$. Since $f$ is a quadric fibration, the moduli space $\CM$ admits a natural structure of a $\mathbb{P}^1$-bundle over a $K3$ surface $X$ which is a double cover of $\mathbb{P}^2$ ramified along a plane sextic $C$. Note that the genericity of $Y$ ensures that both $C$ and $X$ are nonsingular.

The construction of $\CM$ provides a rational map
\[
r: \CM \dashrightarrow F
\]
by projecting a line in $\widetilde{Y}$ to $Y$ via the projection $\widetilde{Y} \rightarrow Y$. Let
\[q: \CM \to X\]
be the $\mathbb{P}^1$-bundle. Hence there is a diagram
\[\begin{tikzcd}
\CM \arrow{d}{q} \arrow[dashed]{r}{r} & F. \\
X
\end{tikzcd}\]
By taking (the closure of) the image of $r$, we obtain a uniruled divisor in $F$,
\[\begin{tikzcd}
D \arrow[dashed]{d}{} \arrow[hook]{r}{} & F. \\
X
\end{tikzcd}\]
The map $r$ induces a morphism of Chow groups
\[
r_\ast : \mathrm{CH}_0(\CM) \rightarrow \mathrm{CH}_0(F)
\]
whose image is given by the classes of $0$-cycles supported on $D$.

\begin{prop}\label{image1}
We have
\[
\mathrm{Im}(r_\ast)_\BQ =
\mathrm{CH}_0(F)_{(0)} \oplus \mathrm{CH}_0(F)_{(2)}.
\]
\end{prop}

\begin{proof}
It is proven in \cite[Theorem 1.2]{CP} that the Chow subgroup of $0$-cycles supported on a uniruled divisor in a holomorphic symplectic variety of $K3^{[d]}$-type is independent of the choice of the uniruled divisor. Hence it suffices to verify the statement for a particular uniruled divisor in~$F$.

In fact, the exceptional divisor of a desingularization of the rational map $\varphi: F \dashrightarrow F$ maps to a uniruled divisor in $F$; see \cite[Proposition 4.4 (b)]{V2}. This uniruled divisor satisfies the proposition by \cite[Proposition 19.5]{SV}; see also \cite[(44)]{V2}.
\end{proof}

Since $q$ is a $\mathbb{P}^1$-bundle, there is an isomorphism of Chow groups 
\begin{equation}\label{Choweqn}
q_* : \mathrm{CH}_0(\CM) \xrightarrow{\sim} \mathrm{CH}_0(X).
\end{equation}
Recall the Beauville--Voisin decomposition
\begin{equation*}
\mathrm{CH}_0(X)_\BQ = \mathrm{CH}_0(X)_{(0)} \oplus \mathrm{CH}_0(X)_{(2)}
\end{equation*}
with $\mathrm{CH}_0(X)_{(0)} = \mathbb{Q} \cdot [o_X]$. The following result provides a link between $0$-cycles on the Fano variety of lines $F$ and the associated $K3$ surface $X$.

\begin{thm}\label{thmcubic}
Let $Y$ be a generic cubic $4$-fold containing a plane, and let $F$ and $X$ be as before. There is a canonical isomorphism
\begin{equation}\label{Chowiso}
\mathrm{CH}_0(F)_\BQ \simeq \mathrm{CH}_0(X)_\BQ \oplus \mathrm{CH}_0(F)_{(4)},
\end{equation}
with
\[
\mathrm{CH}_0(X)_{(0)} \simeq \mathrm{CH}_0(F)_{(0)} \ \ \textrm{and} \ \ \mathrm{CH}_0(X)_{(2)} \simeq \mathrm{CH}_0(F)_{(2)}.
\]
\end{thm}

\begin{proof}
We first prove the isomorphism (\ref{Chowiso}). By Theorem \ref{eigen}, Proposition~\ref{image1}, and the isomorphism (\ref{Choweqn}), it suffices to show the injectivity of \[r_\ast : \mathrm{CH}_0(\CM)_\BQ \rightarrow  \mathrm{CH}_0(F)_\BQ.\] 

For convenience, we work with the moduli space of lines in the $4$-fold $\widetilde{Y}$ and denote it by $\widetilde{F}$. Then $\widetilde{F}$ is birational to $F$ and $\CM$ is a divisor in $\widetilde{F}$,
\[
\widetilde{r} : \CM \hookrightarrow \widetilde{F}.
\]
Since $\mathrm{CH}_0(\widetilde{F})$ is canonically isomorphic to $\mathrm{CH}_0(F)$, we prove instead the injectivity of $\widetilde{r}_\ast : \mathrm{CH}_0(\CM)_\BQ \rightarrow \mathrm{CH}_0(\widetilde{F})_\BQ$.

The Chow group $\mathrm{CH}_0(\CM)$ is spanned by the classes of lines lying in nonsingular quadric fibers of $f: \widetilde{Y} \rightarrow \mathbb{P}^2$. Let $l_1, \ldots, l_k \in \CM$ be a collection of such lines. Assume that they satisfy a nontrivial relation
\begin{equation}\label{relation2}
\sum_{i=1}^k a_i[l_i]=0 \in \mathrm{CH}_0(\widetilde{F})_\BQ.
\end{equation}
We need to show that (\ref{relation2}) also holds in $\mathrm{CH}_0(\CM)_\BQ$. Equivalently, we will prove the relation
\begin{equation}\label{relationK3}
\sum_{i=1}^k a_i[q(l_i)]=0 \in \mathrm{CH}_0(X)_\BQ.
\end{equation}

By abuse of notation, we also write $[l_i] \in \mathrm{CH}_1(\widetilde{Y})$ for the $1$-cycle class of the line $l_i$ in $\widetilde{Y}$. Applying the natural correspondence given by the incidence variety in $\widetilde{F}\times \widetilde{Y}$, we find that the relation (\ref{relation2}) holds in $\mathrm{CH}_1(\widetilde{Y})_\BQ$.

Consider the incidence variety 
\[
\widetilde{\CM}=\{ (l, y)\in \CM \times \widetilde{Y}: y \in l \subset \widetilde{Y}\}
\]
with the projections
\[p_{\CM}: \widetilde{\CM} \rightarrow \CM \ \ \textrm{and} \ \ p_{\widetilde{Y}}: \widetilde{\CM} \rightarrow \widetilde{Y}.\]
We see from the definition that $p_{\CM}$ is a $\mathbb{P}^1$-bundle and $p_{\widetilde{Y}}$ is a double cover. There is a correspondence
\[
\Gamma = {p_{\CM}}_\ast{p_{\widetilde{Y}}}^\ast: \mathrm{CH}_1(\widetilde{Y}) \rightarrow \mathrm{CH}_1(\CM).
\]

Further, let
\[
\tau: X \rightarrow X
\]
be the involution exchanging the two sheets of the double cover $X \rightarrow \mathbb{P}^2$.

\medskip
\noindent{\bf Claim.} Let $l \in \CM$ be a line lying in a nonsingular quadric fiber~$Q$ of $f_Y: \widetilde{Y} \rightarrow \mathbb{P}^2$. Then $\Gamma([l]) \in \mathrm{CH}_1(\CM)$ is the class of the fiber of $q : \CM \to X$ over $\tau(q(l)) \in X$.

\begin{proof}[Proof of the Claim]
By construction, the preimage $p_{\widetilde{Y}}^{-1}(l)$ consists of two components $Z$ and $Z'$ as follows. One is the~locus
\[
Z = \{ (l,y) \in  \CM \times \widetilde{Y}: y \in l \} \subset \widetilde{\CM},
\]
and hence ${p_\CM}_\ast[Z] =0$ in $\mathrm{CH}_1(\CM)$ for dimension reasons. The other component $Z'$ is given by the ruling $\{l_t\}$ in $Q$ where each line $l_t$ meets $l$ at exactly one point $y_t$. It is easy to see that ${p_{\CM}}_\ast[Z']$ is the fiber over $\tau(q(l)) \in X$.
\end{proof}

The injectivity of $\widetilde{r}_\ast : \mathrm{CH}_0(\CM)_\BQ \rightarrow \mathrm{CH}_0(\widetilde{F})_\BQ$ is a direct consequence of the Claim. In fact, starting from the relation (\ref{relation2}) in $\mathrm{CH}_1(\widetilde{Y})_\BQ$, we intersect the class $\Gamma(\sum_{i=1}^k a_i[l_i])$ with a relative ample divisor of $q : \CM \rightarrow X$ and push it forward to~$X$. By the Claim above, we find the vanishing of a $0$-cycle class on $X$, which is a multiple of
\[
\tau_\ast ( \sum_{i=1}^k a_i[q(l_i)] ) \in \mathrm{CH}_0(X)_\BQ.
\]
This proves (\ref{relationK3}), since $\tau_\ast$ is an involution.

Finally, we verify the second part of the theorem. It suffices to prove that $r_\ast$ sends the class of a point on $\CM$, which represents the canonical class $[o_X]$ on $X$ via (\ref{Choweqn}), to the canonical class $[o_F]$ on $F$. 

We choose $x \in \CM$ such that $r$ is well-defined at $x$ and that $q(x) \in X$ lies on a rational curve $R \subset X$. Hence $[q(x)] = [o_X] \in \mathrm{CH}_0(X)$ by \cite{BV}. Then $x$ lies on the rational surface $q^{-1}(R)$ whose image under $r$ is a $2$-dimensional constant cycle subvariety in $F$. By \cite[Proposition 4.5]{V2}, we have
\[[r(x)] = [o_F] \in \mathrm{CH}_0(F).\qedhere\]
\end{proof}


\subsection{Proof of Theorem \ref{cubic4}}

\begin{proof}[Proof of Theorem \ref{cubic4}]
Let $Y$ be a generic cubic $4$-fold containing a plane and let $X$ be the associated $K3$ surface. The morphism
\[C_Y: \mathrm{CH}_0(F)_\BQ \rightarrow \mathrm{CH}_0(X)_\BQ\]
is constructed as the projection to the first factor in the decomposition (\ref{Chowiso}).

More concretely, for $\alpha \in \mathrm{CH}_0(F)_\BQ$, Theorem \ref{thmcubic} implies that there exists a unique $0$-cycle class $\beta \in \mathrm{CH}_0(D)_\BQ$ whose image on $F$ is the projection of $\alpha$ to 
\[
\mathrm{CH}_0(F)_{(0)} \oplus \mathrm{CH}_0(F)_{(2)}.
\]
Then $C_Y(\alpha)$ is the push-forward of $\beta$ to $\mathrm{CH}_0(X)_\BQ$ via the map $D \dashrightarrow X$.

Let $l$ and $l'$ be two lines in $Y$. By Theorem \ref{thmcubic}, we know that
\[C_Y([l]) = C_Y([l']) \in \mathrm{CH}_0(X)_\BQ\]
if and only if
\[[l]_{(2)}= [l']_{(2)} \in \mathrm{CH}_0(F)_{(2)}.\]
Part (ii) of Theorem~\ref{cubic4} then follows from Theorem \ref{thm7}.
\end{proof}

The following corollary is a byproduct of Theorems \ref{cubic4} and \ref{thmcubic}.

\begin{cor}
Let $Y$, $F$, and $X$ be as before. A line $l$ in $Y$ satisfies
\[
[l] = [o_F] \in \mathrm{CH}_0(F)
\]
if and only if
\[
C_Y([l]) \in \mathbb{Q} \cdot [o_X] \subset \mathrm{CH}_0(X)_\BQ.
\]
\end{cor}

\section{Further questions}

\subsection{The dimension two case} \label{dimtwo}
In Section \ref{msso}, we focused on the Beauville--Voisin filtration for moduli spaces of dimension $2d(\mathbf{v})= \mathbf{v}^2+2 >2$. We discuss here the case $\mathbf{v}^2= 0$.

When $\mathbf{v} \in H^\ast(X, \BZ)$ is a primitive Mukai vector satisfying $\mathbf{v}^2=0$, and $\sigma$ is a generic stability condition, the corresponding moduli space $M= M_\sigma(\mathbf{v})$ is a $K3$ surface. Although the Beauville--Voisin filtration on $\mathrm{CH}_0(M)$ is clear by \cite{BV}, its compatibility with the filtration on $\mathbf{D}=D^b(X)$ is not obvious. 

If $M$ is a fine moduli space, then the universal family induces a derived equivalence 
\begin{equation*}\label{equivalence}
D^b(M) \xrightarrow{\sim} D^b(X).
\end{equation*}
Theorem \ref{Huybrechts2} shows that the corresponding isomorphism of Chow groups
\begin{equation*} 
\mathrm{CH}^\ast(M) \xrightarrow{\sim} \mathrm{CH}^\ast(X)
\end{equation*}
preserves the Beauville--Voisin ring. In particular, the canonical class $[o_M] \in \mathrm{CH}_0(M)$ is represented by any object $\CE \in M$ with $c_2(\CE) \in \mathbb{Z} \cdot [o_X]$. The Beauville--Voisin filtration $S_\bullet\mathrm{CH}_0(M)$ indeed comes from the restriction of the filtration $S_\bullet (\mathbf{D})$ on the derived category.

If $M$ is not a fine moduli space, then $D^b(X)$ is equivalent to a derived category of twisted sheaves on $M$,
\[
D^b(M, \alpha) \xrightarrow{\sim} D^b(X).
\]
Recently, Huybrechts showed in \cite[Corollary 2.2]{Motive} that the universal twisted family induces an isomorphism of Chow groups
\begin{equation} \label{isomorphismCH}
\mathrm{CH}^\ast(M)_\BQ \xrightarrow{\sim}  \mathrm{CH}^\ast(X)_\BQ.
\end{equation}
In this case, we also expect that  (\ref{isomorphismCH}) preserves the Beauville--Voisin ring. More generally, we ask the following question.\footnote{By Lemma \ref{Lemma1} (ii), each $\widetilde{S}_i(X)$ has the structure of a cone. Hence it makes sense to extend $\widetilde{S}_\bullet(X)$ to rational coefficients.}


\begin{question}
Does the isomorphism (\ref{isomorphismCH}) preserve O'Grady's filtration~$\widetilde{S}_\bullet$?
\end{question}

One may also ask the same question for arbitrary pairs of twisted $K3$ surfaces which are derived equivalent.

\subsection{More on the Beauville--Voisin filtration} \label{VQ}

Let $M$ be a moduli space of stable objects in $D^b(X)$ as in Section \ref{msso}. Recall the filtration $S_\bullet\mathrm{CH}_0(M)$, where $S_i\mathrm{CH}_0(M)$ is the subgroup spanned by the classes of $\mathcal{E} \in M$ satisfying $c_2(\mathcal{E}) \in S_i(X)$. The following question asks for more precision.

\begin{question} \label{vqest}
For an object $\mathcal{E} \in M$, is it true that \[[\mathcal{E}] \in S_i\mathrm{CH}_0(M)\]
if and only if
\[c_2(\mathcal{E}) \in S_i(X)?\]
\end{question}

By (the proof of) Proposition \ref{presfil}, the answer to Question \ref{vqest} is independent of birational models or modular interpretations.

Question \ref{vqest} for the Hilbert schemes of points on $X$ alone has an interesting interpretation. Let $\gamma \in \mathrm{CH}_0(X)$ be a $0$-cycle class of degree $0$. We may assume
\[\gamma = [\mathrm{Supp}(\xi)] - d[o_X]\]
for some $\xi \in X^{[d]}$ with $d$ sufficiently large.

By an explicit calculation via the motivic decomposition of $X^{[d]}$, we have
\[[\xi] \in S_i\mathrm{CH}_0(X^{[d]})\]
if and only if
\[\gamma^{\times (i + 1)} = 0 \in \mathrm{CH}_0(X^{i + 1}).\]
A positive answer to Question \ref{vqest} for $X^{[d]}$ is then equivalent to the statement~that
\[\gamma \in S_i(X)\]
if and only if
\[\gamma^{\times (i + 1)} = 0 \in \mathrm{CH}_0(X^{i + 1}).\]
The latter is a new characterization of O'Grady's filtration $S_\bullet(X)$ proposed by Voisin.


\begin{thebibliography}{10}


\bibitem{AT} N. Addington, R. P. Thomas, {\em Hodge theory and derived categories of cubic fourfolds.} Duke Math. J. 163 (2014), no. 10, 1885--1927.

\bibitem{BB} A. Bayer, T. Bridgeland, {\em Derived automorphism groups of $K3$ surfaces of Picard rank $1$.} Duke Math. J. 166 (2017), no. 1, 75--124. 

\bibitem{BM1} A. Bayer, E. Macr\`{i}, {\em Projectivity and birational geometry of Bridgeland moduli spaces.} J. Amer. Math. Soc. 27 (2014), no. 3, 707--752.

\bibitem{BM} A. Bayer, E. Macr\`{i}, {\em MMP for moduli of sheaves on $K3$s via wall-crossing: nef and movable cones, Lagrangian fibrations.} Invent. Math. 198 (2014), no. 3, 505--590.

\bibitem{B} A. Beauville, {\em Vari\'{e}t\'{e}s k\"{a}hleriennes dont la premi\`{e}re classe de Chern est nulle.}  J. Differential Geom. 18 (1983), no. 4, 755--782 (1984).

\bibitem{B2} A. Beauville, {\em On the splitting of the Bloch--Beilinson filtration.} Algebraic cycles and motives. Vol. 2, 38--53, London Math. Soc. Lecture Note Ser., 344, Cambridge Univ. Press, Cambridge, 2007.

\bibitem{BD} A. Beauville, R. Donagi, {\em La vari\'{e}t\'{e} des droites d'une hypersurface cubique de dimension $4$.} C. R. Acad. Sci. Paris S\'er. I Math. 301 (1985), no. 14, 703--706.

\bibitem{BV} A. Beauville, C. Voisin, {\em On the Chow ring of a $K3$ surface.} J. Algebraic Geom. 13 (2004), no. 3, 417--426.

\bibitem{Br} T. Bridgeland, {\em Stability conditions on triangulated categories.} Ann. of Math. (2) 166 (2007), no. 2, 317--345.

\bibitem{Br1} T. Bridgeland. {\em Stability conditions on $K3$ surfaces.} Duke Math. J. 141 (2008), no. 2, 241--291.


\bibitem{CP} F. Charles, G. Pacienza, {\em Families of rational curves on holomorphic symplectic varieties and applications to $0$-cycles.} arXiv:1401.4071v2.

\bibitem{Ha} B. Hassett, {\em Special cubic fourfolds.} Compositio Math. 120 (2000), no. 1, 1--23.

\bibitem{Huy} D. Huybrechts, {\em Chow groups of $K3$ surfaces and spherical objects.} J. Eur. Math. Soc. 12 (2010), no. 6, 1533--1551.

\bibitem{H2} D. Huybrechts, {\em Curves and cycles on $K3$ surfaces. With an appendix by C.~Voisin.} Algebr. Geom. 1 (2014), no. 1, 69--106.

\bibitem{Motive} D. Huybrechts, {\em Motives of isogenous $K3$ surfaces.} Comment. Math. Helv., to appear.

\bibitem{HL} D. Huybrechts, M. Lehn, {\em The geometry of moduli spaces of sheaves. Second edition.} Cambridge Mathematical Library, Cambridge Univ. Press, Cambridge, 2010, xviii+325 pp.



\bibitem{Kuz} A. Kuznetsov, {\em Derived categories of cubic fourfolds.} Cohomological and geometric approaches to rationality problems, 219--243, Progr. Math., 282, Birkh\"auser Boston, Inc., Boston, MA, 2010.

\bibitem{Mac} C. Maclean, {\em Chow groups of surfaces with $h^{2, 0} \leq 1$.} C. R. Math. Acad. Sci. Paris 338 (2004), no. 1, 55--58.

\bibitem{MS} E. Macr\`{i}, P. Stellari, {\em Fano varieties of cubic fourfolds containing a plane.} Math. Ann. 354 (2012), no. 3, 1147--1176.

\bibitem{MZ} A. Marian, X. Zhao, {\em On the group of zero-cycles of holomorphic symplectic varieties.} arXiv preprint, 2017.




\bibitem{Muk1} S. Mukai, {\em On the moduli space of bundles on $K3$ surfaces. I.} Vector bundles on algebraic varieties (Bombay, 1984), 341--413, Tata Inst. Fund. Res. Stud. Math., 11, Tata Inst. Fund. Res., Bombay, 1987.



\bibitem{OG2} K. G. O'Grady, {\em Moduli of sheaves and the Chow group of $K3$ surfaces.} J. Math. Pures Appl. (9) 100 (2013), no. 5, 701--718.

\bibitem{R} U. Rie{\ss}, {\em On the Chow ring of birational irreducible symplectic varieties.} Manuscripta Math. 145 (2014), no. 3-4, 473--501.

\bibitem{Roi} A. A. Rojtman, {\em The torsion of the group of $0$-cycles modulo rational equivalence.} Ann. of Math. (2) 111 (1980), no. 3, 553--569.


\bibitem{SV} M. Shen, C. Vial, {\em The Fourier transform for certain hyperk\"ahler fourfolds.} Mem. Amer. Math. Soc. 240 (2016), no. 1139, vii+163 pp. 


\bibitem{V0} C. Voisin, {\em Intrinsic pseudo-volume forms and $K$-correspondences.} The Fano Conference, 761--792, Univ. Torino, Turin, 2004. 

\bibitem{V} C. Voisin, {\em On the Chow ring of certain algebraic hyper-K\"{a}hler manifolds.} Pure Appl. Math. Q. 4 (2008), no. 3, part 2, 613--649.

\bibitem{V1} C. Voisin, {\em Rational equivalence of $0$-cycles on $K3$ surfaces and conjectures of Huybrechts and O'Grady.} Recent advances in algebraic geometry, 422--436, London Math. Soc. Lecture Note Ser., 417, Cambridge Univ. Press, Cambridge, 2015.

\bibitem{V2} C. Voisin, {\em Remarks and questions on coisotropic subvarieties and $0$-cycles of hyper-K\"ahler varieties.} $K3$ surfaces and their moduli, 365--399, Progr. Math., 315, Birkh\"auser/Springer, 2016.

\end{thebibliography}
\end{document}